\documentclass[12pt]{amsart}
\usepackage{graphicx}
\usepackage{color}
\usepackage{setspace}
\usepackage{eufrak}
\usepackage{float}
\usepackage{etex}
\usepackage[all]{xy}
\usepackage{amssymb, amsmath}
\usepackage{ascmac}

\def\qed{\hfill $\Box$}

\def\qed{\hfill $\Box$}

\def\b0{\mbox{\boldmath $0$}}

\newtheorem{proposition}{Proposition}[section]
\newtheorem{theorem}{Theorem}[section]

\newtheorem{definition}{Definition}[section]
\newtheorem{remark}{Remark}[section]

\newtheorem*{Thm}{Theorem}

\begin{document}

\title[Cr-invariants for surfaces in $4$-space]{Cr-invariants for Surfaces in $4$-space}

\author[J. L. ~Deolindo Silva]{Jorge Luiz Deolindo Silva} 
    \address[J. L. ~Deolindo Silva]{Departamento de Matem\'atica, 
                   Univerisidade Federal de Santa Catarina, 
                   CEP 89036-004, Blumenau, SC, Brasil }


\thanks{The author was totally supported by FAPESP grants no.2012/00066-9. }

\begin{abstract}
 We establish cross-ratio invariants for surfaces in 4-space in an analogous way to Uribe-Vargas's work for surfaces in 3-space.  We study the geometric locii of local and multi-local singularities of ortogonal projections of the surface. The  cross-ratio invariants at $P_3(c)$-points  are used to recover two moduli in the 4-jet of the projective parametrization of the surface and show the stable configurations of asymptotic curves. 
\end{abstract}

\maketitle
\setlength{\baselineskip}{16pt}

\section{Introduction}

The study of the differential geometry of immersed surfaces in 4-space has been of great interest for the past 50 years. In this paper, we study the geometry of a smooth surface in {4-space (Euclidean, affine or projective)} which is associated to its contact with lines, applying singularity theory techniques to this subject.  
In this sense, has been studied by many author: \cite{bruce-nogueira,DK,little,dmm,projsrfR4,OsetTari}. 
The contact of a surface with lines is described by the $\mathcal A$-singularities of the family of orthogonal projections to $3$-spaces.   (Two germs $f$ and $g$ are said to be $\mathcal A$-equivalent and write $f\sim_{\mathcal A} g$, if $g=k\circ f\circ h^{-1}$ for some germs of diffeomorphisms $h$ and $k$ of, respectively, the source and target.) 
When projecting the surface  along an asymptotic direction, at isolated points in the parabolic set, the singular projection may have a $P_3(c)$-point. 
The $P_3(c)$-point is a singularity of an orthogonal projection which is $\mathcal A$-equivalent to the germ $(x,xy+y^3,x^2y+cy^4)$ where $c\in\mathbb R$ with $c\neq0,1/2,1,3/2$ is a modulus parameter.

In \cite{r4surf,projsrfR4} is observed that the $P_3(c)$-point have a geometric behavior similar to the cusp of Gauss of a smooth surface in $\mathbb R^3$. In fact, the equation of asymptotic curves of a surface in $\mathbb R^3$ (resp. $\mathbb R^4$) has fold singularity at a cusp of Gauss (resp. at $P_3(c)$-point).
 The cusps of Gauss 
are studied in  \cite{cuspGauss,bgt1,bgt2,DKOS,oliver09,uribeInv}. In \cite{uribeInv}, Uribe-Vargas introduced a cross-ratio invariant ({\it cr-invariant}) at cusps of Gauss using 
the parabolic, flecnodal and conodal curves of the surface. 
If the surface in $\mathbb R^3$ is given by Monge form $z=f(x,y)$ at a cusp of Gauss, then  the $4$-jet of $f$ can be taken in the form 
$
\frac{y^2}{2}-x^2y+{\lambda} x^4, \;\;\lambda\neq0,\frac{1}{2}, 
$  
  see \cite{Platonova,DKOS}. Uribe-Vargas showed that the cr-invariant
is equal to  $2\lambda$. The cr-invariant describes also the relative position and configurations of parabolic, flecnodal and conodal curves on the surface. 

We carry out a similar approach to that in \cite{uribeInv} for generic surfaces in 4-space at $P_3(c)$-points.  
We obtain the curves of generic local and multi-local singularities of orthogonal projection at a $P_3(c)$-point. We use the projective classification of $4$-jet of the Monge form of germ of surface $M$ in $\mathbb P^4$ at $P_3(c)$-point in  \cite{DK}  given by  
\[
(x^2+xy^2+{ \alpha} y^4,xy+{ \beta} y^3+\phi)
\]
where  $\phi$ is a polinomial in $x,y$ of degree $4$, and  by considering the different cross-ratios of groups of 4 amongst these curves, we obtain some projective invariants (cr-invariants) of the surface at such a point. We get 
 the following result (see also $\S$ \ref{6}).
\begin{Thm}\label{main}
Let  $M$ be a smooth immersed surface in $\mathbb R^4$ given by  the Monge form as above at a $P_3(c)$-point. Then the moduli ${ \alpha}$ and ${ \beta}$ can be written as function of the cross-ratios invariants (cr-invariants) at  $P_3(c)$-point.
\end{Thm}

In Section \ref{3} we determine the adjacences of the $P_3(c)$-singularity to multi-germs of $\mathcal A_e$-codimension $2$ which are responsible for finding the local and multi-local curves at $P_3(c)$-point on surface. 
There are several classifications and $\mathcal A$-adjacences of germs and multi-germs of $(\mathbb R^2,S)\to(\mathbb R^3,0)$  (see \cite{kirk,roberta,mond}), however there are not $\mathcal A$-adjacence of $P_3(c)$ to multi-germs.  In Section \ref{4}, we show the existence of three multi-local curves and two local curves at $P_3(c)$-point, where one of them is obtained from the  $2$-jet of the inflections of the asymptotic curves that we call {\it flecnodal curve} in Section \ref{5}. The asymptotic curves are given using {\it binary differential equation} (BDE) that widely appear in geometric problems.
Finally in  Section \ref{6}, we introduce the cr-invariant and as a consequence of main Theorem we establish the stable configurations of these asymptotic, flecnodal and local curves as well as the relative position of the local and multi-local curves using the parameters $\alpha$ and $\beta$ given by the projective classification.

\section{Preliminary}\label{2}

We briefly review and establish some notation concerning a smooth surface $M$ in $\mathbb R^4$. Several authors have studied the differential geometry of generic surfaces in $\mathbb R^4$ (see for example \cite{bruce-nogueira,r4surf,DK,Ronaldoetal,little,dmm,projsrfR4,OsetTari}). Let $M$ be a regular surface in the Euclidean space $\mathbb R^4$. Consider $p\in M$ and an unit circle in $T_pM$ parametrized by $\theta \in [0,2\pi]$. The \emph{curvature ellipse} $\eta(\theta)$ in the normal plane $N_pM$, is the image of this unit circle described by a pair of quadratic forms $(Q_1, Q_2)$ (\cite{little}). Points on the surface are classified according to the position of the point $p$ with
respect to the ellipse ($N_pM$ is viewed as an affine plane through $p$). The point $p$ is
called \emph{elliptic/parabolic/hyperbolic} if it is inside/on/outside the ellipse at $p$, respectively.

The pair of quadratic forms is the $2$-jet of the
1-flat map $F:{({\mathbb R}}^2,0)\to {{(\mathbb R}}^2,0)$ (i.e. without constant or linear terms) whose graph,
in orthogonal coordinates, is locally the surface $M$. Using a different approach to the geometry of surfaces in ${{\mathbb R}}^4$ given in \cite{bruce-nogueira}, each point on the surface determines a pair of quadratics:
\[
(Q_1,Q_2)=(ax^2+2bxy+cy^2,lx^2+2mxy+ny^2). 
\]

Representing a binary form $Ax^2+2Bxy+Cy^2$ by its coefficients $(A:B:C)\in {\mathbb R}P^2$, there
is a cone $\Gamma: B^2-AC=0$ representing the perfect squares.
If the forms $Q_1$ and $Q_2$ are
independent, then they determine a line in the projective plane ${\mathbb R} P^2$ and the cone a
conic. This line meets the conic $\Gamma$ in 0, 1 or 2 points according as
$\delta(p)<0,=0,>0$ where \[ \delta(p)=(an-cl)^2-4(am-bl)(bn-cm).\]
A point $p$ is said to be \emph{elliptic/parabolic/hyperbolic} if $\delta(p)<0/=0/>0$. The set of points in $M$
 when $\delta=0$ is called the \emph{parabolic set} of M and is denoted by $\Delta$.
If $Q_1$ and $Q_2$ are dependent at a point $p$, then the point is called \emph{inflection point}.

There is an action of ${\mathcal G}=GL(2,{\mathbb R})\times GL(2,{\mathbb R})$ on pairs of binary forms.
The ${\mathcal G}$-orbits are as follow (see for example \cite{gibson}):

\[
\begin{array}{cl}
(x^2,y^2) & \mbox{hyperbolic point} \\
(xy,x^2-y^2) & \mbox{elliptic point}\\
(x^2,xy) &\mbox{parabolic point}\\
(x^2\pm y^2,0)&\mbox{inflection point}\\
(x^2,0)&\mbox{degenerate  inflection point}\\
(0,0)&\mbox{degenerate inflection point}\\
\end{array}
\]

The geometrical characterization of points on $M$ using singularity theory is carried out in \cite{dmm} via the family of height functions
$h : M\times S^3 \to {{\mathbb R}} $ given by $h(p,v)=\langle p,v\rangle$,
 where $S^3$ denotes the unit sphere in ${\mathbb R}^4$. For $v$ fixed the height function $h_v(p)=h(p,v)$ is singular if and only
if $v\in N_pM$. It is shown in \cite{dmm} that elliptic points are non-degenerate critical points of
$h_v$ for any $v\in N_pM$. At a hyperbolic point, there are exactly two directions in $N_pM$, labelled \emph{binormal directions},
such that $p$ is a degenerate critical point of the corresponding height functions. The two binormal directions coincide at a parabolic point.

The direction of the kernel of the Hessian of the height functions along a binormal
direction is an \emph{asymptotic direction} associated to the given binormal direction (\cite{dmm}).
The asymptotic directions are labelled by conjugate directions in \cite{little}, and are defined as the directions along $\theta$ such that the curvature vector $\eta(\theta)$ is tangent to the curvature
ellipse (see also (\cite{Ronaldoetal,dmm})). So if $p$ is not an inflection point, then there are $2/1/0$ asymptotic
directions at $p$ depending on $p$ being a hyperbolic/parabolic/elliptic point. The generic configurations of the asymptotic curves are given in \cite{r4surf}.

Asymptotic directions can also be described via the singularities
of the projections of $M$ to hyperplanes (see \cite{bruce-nogueira}). The family of projections is given by

\[
\begin{array}{rcl}
P : M\times S^3& \to& TS^3 \\
(p,v)&\mapsto&(v,p-\langle p,v\rangle v).
\end{array}
\]
For $v$ fixed, the projection can be viewed locally at a point $p\in M$ as a map germ $P_v:{(\mathbb R}^2,0)\to
({\mathbb R}^3,0)$. If we allow smooth changes of coordinates in the source and target
(i.e. consider the action of  group $\mathcal{A}$) then the generic $\mathcal A$-singularities of $P_v$
are those that have $\mathcal A_e$-codimension less than or equal to 3 (which is the dimension of $S^3$). These are listed in Table 1 (see  \cite{mond}).
\begin{table}[htp]
 \caption{The generic local singularities of projections of $M$ to hyperplanes.}\label{mond}
\begin{tabular}{ccc}
\hline
Name & Normal form & ${\mathcal A}_e$-codimension\cr
\hline
Immersion & $(x,y,0)$ &0\cr
Cross-cap & $(x,y^2,xy)$ &0\cr
$B^{\pm}_{k}$ & $(x, y^2, x^2y \pm y^{2k+1}), k=2,3$ & $k$ \cr
$S^{\pm}_{k}$ & $(x, y^2, y^3 \pm x^{k+1}y), k=1,2,3$ & $k$\cr
$C^{\pm}_{k}$ & $(x, y^2, xy^3 \pm x^{k}y), k=3$ & $k$\cr
$H_k$ & $(x,xy+y^{2k+2},y^3), k=2,3$ & $k$\cr
$P_3(c)$ & $(x, xy + y^3, xy^2 + cy^4),\;c\neq0,\frac{1}{2},1,\frac{3}{2}$ & $3^*$\cr
\hline
\end{tabular}
\hspace{10cm}{ $ $ $^*$ The codimension of $P_3(c)$ is that of its stratum.}
\end{table}

The projection $P_v$ is singular at $p$ if and only if $v\in T_pM$. The singularity is
a cross-cap unless $v$ is an asymptotic direction at $p$. As $P$ has  3-parameters, the ${\mathcal A}_e$-codimension $2$ singularities
occur generically on curves on the surface and the ${\mathcal A}_e$-codimension $3$ ones at special points
on these curves (see Figure \ref{fig:SpecialC} for their configurations). The
$H_2$-curve coincides with the $\Delta$-set (\cite{bruce-nogueira}). The $B_2$-curve of $P_v$, with $v$ asymptotic, is
also the $A_3$-set of the height function along the binormal direction associated to $v$
(\cite{bruce-nogueira}). This curve meets the $\Delta$-set tangentially at a $P_3(c)$-singularity of projection along the unique asymptotic direction (\cite{r4surf}) and intersects
the $S_2$-curve transversally at a $C_3$-singularity. At inflection points the $\Delta$-set has a Morse singularity. For more details about inflection points see \cite{bruce-nogueira}.

\begin{figure}[htb]
\includegraphics[ width=11.5cm, height=5.5cm]{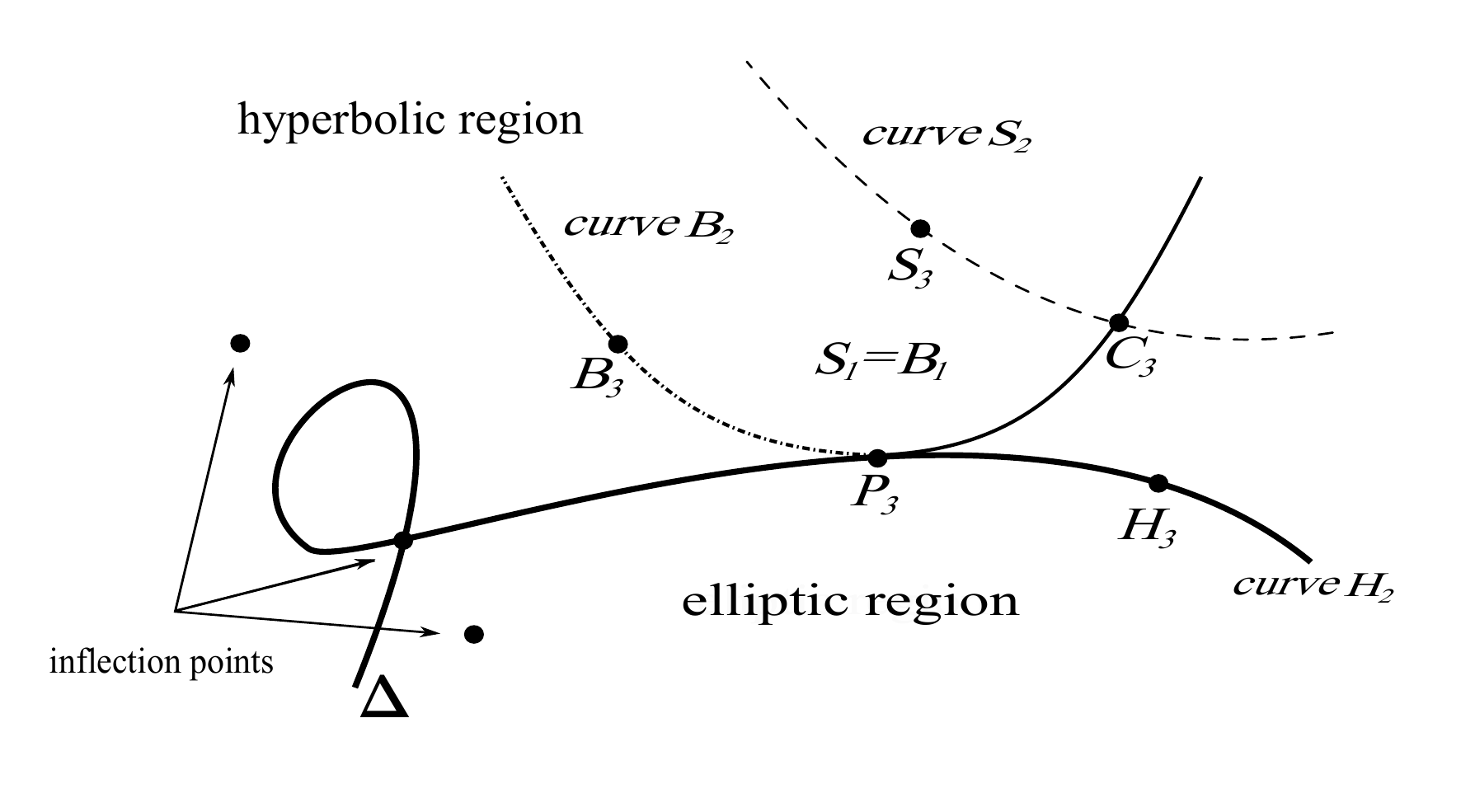}
\caption{Curves and special  points on $M$ (\cite{dmm,projsrfR4}).}
\label{fig:SpecialC}
\end{figure}

 The multi-local singularities of orthogonal projections also give geometric information of the surface in $\mathbb R^4$.  
We determine in Section \ref{3} and Section \ref{4} the local and multi-local singularities of $\mathcal A_e$-codimension 2 of the orthogonal projections that are adjacent to the $P_3(c)$-singularity. 

In fact, consider a normal form  $f:(\mathbb{R}^2,0)\to  (\mathbb{R}^3,0$) of a $\mathcal A_e$-codimension $3$ singularity and 
$
F:(\mathbb{R}^2 \times \mathbb{R}^3,0) \to (\mathbb{R}^3,0)
$ 
the   $\mathcal A_e$-versal deformation of $f$.   The {\it singular set} of $F$ is given by 
\[
C(F)=\{(\textbf{x},\textbf{v})\in\mathbb{R}^2 \times\mathbb{R}^3| rank (d_{\textbf{x}}F(\textbf{x},\textbf{v}))<2 \},
\]
where  $d_{\textbf{x}}F(\textbf{x},\textbf{v})$ is the differentiation of $F$ with respect  to $\textbf{x}$ at $(\textbf{x},\textbf{v})$. 

The family of orthogonal projection  $P:(M  \times  S^3,(0,0)) \to \mathbb{R}^3$ is generically an $\mathcal A_e$-versal deformation  of a local singularity of $P_v$ of $\mathcal A_e$-codimension 3.  If $f\sim_{\mathcal{A}}P_v$ and $F$ and $P$ are an $\mathcal A_e$-versal deformation of $f$ and $P_v$, respectively, then there is a germ of diffeomorphism  
such that $F\sim_{\mathcal{A}} P$. Consequently  $C(F)$ is diffeomorphic to  $C(P)$. 
Thus the geometric conditions on the singular set  $C(F)$ that describe the local and multi-local singularities of  $\mathcal A_e$-codimension 2 (that appear of the adjacency) are sent by diffeomorphism that ensure the existence of the local and multi-local curves  $\beta$ in a singular set of $C(P)$.
Hence, the canonical projection $\pi_1:(M  \times  S^3,(0,0))\rightarrow (M,0)$ of $\beta$ is a curve that represent the local and multi-local singularity of  $\mathcal A_e$-codimension 2 of $P_v$ on surface  $M$. In the next section we show some results about adjacency.

\section{Adjacency}\label{3}
Let $L$ and $K$ be two $\mathcal{A}$-classes of germs of $(\mathbb{R}^2,0)\to(\mathbb{R}^3,0)$. We say that $L$ is {\it adjacent} to $K$, denoted by  $(L\rightarrow K)$, if for all $f \in L$ there is a deformation $F_u$ of $f$ such that  $F_u \in K$ for some $u$ near to zero.  We show the  $\mathcal A$-adjacencies of  the germ $P_3(c)$ to $\mathcal A_e$-codimension $2$ bi-germs of $(\mathbb{R}^2,\{s_1,s_2\})\to(\mathbb{R}^3,0)$. 

For the local case, the $\mathcal A$-adjacencies are given by Mond (\cite{mond},  for instance $P_3(c)\to S_2,\;B_2\;,H_2$). 
Also for the multi-local case, the $\mathcal A$-adjancencies are studied in \cite{kirk,roberta}. However, the adjacencies of $P_3(c)$ to multi-germs are not considered previously.  

 The bifurcation diagrams of multi-germs and geometric recognition criteria of multi-germs are described in  \cite{kirk}. We restrict our study to bi-germs in Table \ref{bigerm} (see Remark \ref{tri}). 
 
\begin{table}[h]
\caption{ Bi-germs of $\mathcal{A}_e$-codimension 2 of projection of $M$ to hyperplanes.}\label{bigerm}
\begin{tabular}{ccc}
\hline
Name & Normal Form &${\mathcal A}_e$-codimension\cr
\hline
$[A_2]$ & $(x,y,0;X,Y,X^2+Y^3)$ &$2$\cr
$(A_0S_0)_2$ & $(x,y,0;Y^2,XY+Y^5,X)$ &$2$\cr
$A_0S_1^{\pm}$ & $(x, y,0;Y^3\pm X^2Y,Y^2,X)$ & $2$ \cr
$A_0S_0|A_1^{\pm}$ & $(x, y,0, X,XY,Y^2 \pm X^2)$ & $ 2$ \cr
\hline
\end{tabular}
\end{table}
Geometric criteria:
\begin{itemize}

\item[(a)] $[A_2]$: {\it two regular surfaces with tangency of type $A_2$.}

\item[(b)] $(A_0S_0)_2$: {\it regular surface transversal to a cross-cap but tangent to the  double points curve at the cross-cap point.}

\item[(c)] $A_0S_1^{\pm}$: {\it surface with singularity $S_1^{\pm}$ intersecting transversally another regular surface.}

\item[(d)]  $A_0S_0|A_1^{\pm}$: {\it regular surface tangent to a cross-cap but transversal to the double points curve at the cross-cap point.}

\end{itemize}

\begin{remark}\label{tri}{\normalfont
We do not study other multi-germs $(\mathbb{R}^2,S)\to(\mathbb{R}^3,0)$ because $P_3(c)$ is not adjacent to $\mathcal{A}_e$-codimension 2 tri-germs $A_0^3|A_2$ and $(A_0^2|A_1^{\pm})A_0$ according to Proposition   \ref{condA0S1} below. Moreover, there are at most triple-point  in a versal unfolding of $P_3(c)$ (see \cite{marar}). Therefore we do not have adjacencies to $\mathcal{A}_e$-codimension 2 four-germs and five-germs.}
\end{remark}

We use the recognition criteria of the singularities  $S_0$ and $S_1^{\pm}$ of 
Saji in \cite{Kentaro} for the bi-germ $A_0S_1^{\pm}$, $(A_0S_0)_2$ and $A_0S_0|A_1^{\pm}$. 

\begin{theorem} \label{criteria} {\normalfont\cite{Kentaro}}
If  $f:(\mathbb{R}^2,0)\rightarrow(\mathbb{R}^3,0)$ has corank $1$ at $0$, then there are germs vector fields $\xi$ and $\eta$ at the origin such that $df_0(\eta_0)=0$ and $\xi_0,\eta_0\in T_0\mathbb{R}^2$ are linearly independent, where  $\xi_0,\eta_0$ is the values of the fields at the origin. Set $\varphi=det(\xi f,\eta f,\eta\eta f)$.
\begin{itemize}
\item[(i)] The germ  $f$ has an $S_0$-singularity if $\xi\varphi\neq0$.
\item[(ii)] If $\xi\varphi=0$, then $f$  is $\mathcal{A}$-equivalent to $S_1^-$ at $0$ if and only if  $\varphi$ has a critical point at $0$ and  $\det(Hess \varphi(0)) > 0$. On the other hand, $f$ at $0$ is $\mathcal{A}$-equivalent to $S_1^+$ if and only if  $\varphi$ has critical point at $0$, $\det(Hess \varphi(0)) < 0$ and the vectors  $\xi f(0)$ and $\eta\eta f(0)$ are linearly independent.
\end{itemize}
\end{theorem}

\begin{proposition}\label{condA0S1} The singularity  $P_3(c)$ is adjacent to the bi-germ $A_0S_1^{\pm}$, $(A_0S_0)_2$ and $A_0S_0|A^{\pm}_1$.
\end{proposition}
\begin{proof} Consider the  $\mathcal A_e$-versal deformation of  $P_3(c)$ given by  
\[
F(x,y,u)=(x,xy+y^3+ay,xy^2+cy^4+by+dy^3),
\]
with $c\neq0,1/2,1,3/2$ and $u=(a,b,c,d)$. The singularity $A_0S_1^{\pm}$  appears in the family $F$ if and only if  
\begin{itemize}
  \item[(i)] $F(x_1,y_1,u)=F(x_2,y_2,u)$ for $(x_1,y_1)\neq(x_2,y_2)$ and
  \item[(ii)] $(x_2,y_2)$ is a point of type $S_1$.
  \end{itemize}
  Using Theorem \ref{criteria}, we consider
  \begin{center}  
    $\varphi=$ \begin{tabular}{|ccc|}
$1$ & $y$ & $y^2$ \cr
$0$ & $x+3y^2+a$ &$2xy+4cy^3+b+3dy^2$ \cr
$0$ & $6y$ & $2x+12cy^2+6dy$ \cr
\end{tabular}.
\end{center}
Then $\displaystyle\frac{\partial\varphi}{\partial x}(x,y)=4x+12cy^2+6dy-6y^2+2a$. Thus, at $(x_2,y_2)$, $\displaystyle\frac{\partial\varphi}{\partial x}(x_2,y_2)$ gives  
\begin{equation}\label{eq*}	
4x_2+12c y_2^2+6dy_2-6y_2^2+2a=0.
\end{equation}

As $(x_2,y_2)$ is a singular point, so $a=-x_2-3y^2_2$ and $b=-(2x_2y_2+4c y_2^3+3dy_2^2)$. Substituting in (\ref{eq*}) we obtain 
\[
x_2=-6c y^2_2-3dy_2+6y^2_2.
\]
Substituting the expressions of $a$, $b$ and $x_2$  in $F(x_1,y_1,u)=F(x_2,y_2,u)$ we have:
\[
\left\{\begin{array}{l}
  x_1=x_2;\\
   -(2y_2+y_1)(y_2-y_1)^2=0 \Rightarrow y_1=-2y_2;\\
  27y_2^3(c y_2+2y_2+d)=0\Rightarrow d=(-c +2)y_2.
\end{array}\right.
\]
Therefore by substituting  $d$ in the expressions of  $b$ and $x_2$ we obtain a curve  
\[
\{(-3c y_2^2,-2y_2,(3c-3)y_2^2,(5c-6)y_2^3,(-c+2)y_2)\;|\;y_2\in\mathbb R\}
\]
in $(\mathbb{R}^2\times\mathbb{R}^3,0)$.  The projection of the above curve by  $\pi:(\mathbb{R}^2\times\mathbb{R}^3,0)\to(\mathbb{R}^2,0)$ is the curve  $(-3c y_2^2,-2y_2)$  and gives the locus of the $A_0S_1^{\pm}$-singularity at the $P_3(c)$-singularity. 

The calculations to $(A_0S_0)_2$ are similarly obtained. The conditions
 for bi-local singularity  $(A_0S_0)_2$ are as following:    
 \begin{itemize}
  \item[(i)] $F(x_1,y_1,u)=F(x_2,y_2,u)$ for $(x_1,y_1)\neq(x_2,y_2)$;
  \item[(ii)] The point  $(x_2,y_2)$ is a cross-cap;
  \item[(iii)] The limiting tangent vector to the double points curve of the cross-cap belongs to tangent space of the image of $F_u$ at the point $(x_1,y_1)$.    
\end{itemize}
 We obtain the parametrized curve by  
\[
\{(-3c y_2^2,-2y_2,(3c-3)y_2^2+\cdots,(14c-6)y_2^3+\cdots,(-4c+2)y_2+\cdots)\;|\;y_2\in\mathbb R\}  
\]
in $(\mathbb{R}^2\times\mathbb{R}^3,0)$. The projection of the above curve by  $\pi:(\mathbb{R}^2\times\mathbb{R}^3,0)\to(\mathbb{R}^2,0)$ is the curve  $(-3c y_2^2,-2y_2)$ and gives the locus of the  $(A_0S_0)_2$-singularity at the $P_3(c)$-singularity.   

The conditions for bi-local singularity  $A_0S_0|A^{\pm}_1$ are:
\begin{itemize}
  \item[(i)] $F(x_1,y_1,u)=F(x_2,y_2,u)$ for $(x_1,y_1)\neq(y_2,y_2)$;
  \item[(ii)] The point  $(x_2,y_2)$ is a  cross-cap;
  \item[(iii)]  The tangent vector of the cross-cap belongs to tangent space of image of $F_u$ at point $(x_1,y_1)$. 
  \end{itemize} 
We obtain the parametrized curve by
\[
\{(-3c y_2^2,-2y_2,(3c-3)y_2^2,(-4c+3)y_2^3,(2c-1)y_2)\;|\;y_2\in\mathbb R\}
\]
in $(\mathbb{R}^2\times\mathbb{R}^3,0)$. The projection of the above curve by  $\pi:(\mathbb{R}^2\times\mathbb{R}^3,0)\to(\mathbb{R}^2,0)$  is the curve $(-3c y_2^2,-2y_2)$ and gives the locus of the    $A_0S_0|A^{\pm}_1$-singularity at the $P_3(c)$-singularity. \qed
\end{proof}

\section{The family of orthogonal projections}\label{4}

\indent\indent In this section, 
we show the existence of generic curves on the surface $M$ given by local and multi-local singularities of ortogonal projections at $P_3(c)$-point. 
We next establish some notation that will be used throughout the paper. Consider the family of orthogonal projections in $\mathbb{R}^4$ given by $P:\mathbb{R}^4\times S^3\rightarrow TS^3$ and  fixed ${\bf u}$ a directions. 
The map  $P_{\bf u}$ has singularity worse than a  cross-cap, when  ${\bf u}$ is a asymptotic direction. Moreover, at hyperbolic (parabolic) points  we have two (one resp.) asymptotic directions.

We choose local coordinates at $p$ so that the surface is given in Monge form
\[
(x, y, f^1(x, y), f^2(x, y))
\]
where $(j^1f^1(0, 0), j^1f^2(0, 0))=(0,0)$.
We denote by $(X,Y,Z,W)$ the coordinates in  $\mathbb R^4$ and we fix ${\bf u}=(0,1,0,0)$ as an asymptotic direction at the origin. We parametrize the directions near $\bar{u}$ by $(u,1,v,w)$. Instead of the orthogonal projection to the plane $(u,1,v,w)^\perp$, we project to the fix plane $(X,Z,W)$. The modified family of projections is given by  
\[
\begin{array}{cccl}
                 P: & (\mathbb{R}^2\times \mathbb{R}^3,0) & \rightarrow & (\mathbb{R}^3,0) \\
                  & ((x,y),(u,v,w)) & \mapsto & P_{\textbf{u}}=(x-uy,f^1(x,y)-vy,f^2(x,y)-wy),
               \end{array}
\]
with $P_0(x,y)=(x,f^1(x,y),f^2(x,y))$. The $P_3(c)$-singularity  occurs at parabolic points. In this cases we can take 
\[
\begin{array}{l}
f^1(x,y)=x^2+\sum_{i=0}^{3}a_{3i}x^{3-i}y^i+\sum_{i=0}^{4}a_{4i}x^{4-i}y^i+\sum_{i=0}^{5}a_{5i}x^{5-i}y^i+o_5(x,y) \\ 
f^2(x,y)=xy+\sum_{i=0}^{3}b_{3i}x^{3-i}y^i+\sum_{i=0}^{4}b_{4i}x^{4-i}y^i+\sum_{i=0}^{5}b_{5i}x^{5-i}y^i+o_5(x,y).
 \end{array}
               \]

\begin{proposition}\label{CondPar}  With the above notation,  the conditions for the projection $P_0$  to have a  $P_3(c)$-singularity are 
\[
   a_{33} = 0,\; b_{33}a_{32}a_{44}\neq0 \mbox{ and  $\displaystyle\frac{a_{44}}{b_{33}a_{32}}\neq 0,\frac12,1,\frac32.$}
\]
The  $P_3(c)$-singularity is versally unfolded by family of the orthogonal projections  $P$ if and only if $5a_{32}b_{33}-6b^{2}_{33}-4a_{44} \neq 0.$
\end{proposition}
\begin{proof} The proof follows by standard singularity theory calculations and is omitted.  \qed \end{proof}

\vspace{0.5cm}

Consider the parabolic curve $\Delta$ given by $\delta=0$ (see $\S$ 2), where
\[
\delta=(an-cl)^2-4(am-bl)(bn-cm)\mbox{}.
\]
The $2$-jet of the coefficients of  $a,b,c,l,m$ and $n$ are given as following 
\[
\begin{array}{rll}
a & =& \frac{1}{2}f^1_{xx}=1+ 3 a_{30} x + a_{31} y +6 a_{40} x^2+ 3a_{41} xy + a_{42} y^2,    \\
b &= &\frac{1}{2}f^1_{xy}=a_{31}x+a_{32}y+\frac32a_{41}x^2+2a_{42}xy+\frac32a_{43}y^2,\\
c & = &\frac{1}{2}f^1_{yy}=a_{32}x+3a_{33}y+a_{42}x^2+3a_{43}xy+6a_{44}y^2,\\
l  &= &\frac{1}{2}f^2_{xx}=3b_{30}x+b_{31}y+6b_{40}x^2+3b_{41}xy+b_{42}y^2,\\
m &= &\frac{1}{2}f^2_{x y}=\frac12+b_{31}x+b_{32}y+\frac32b_{41}x^2+2b_{42}xy+\frac32b_{43}y^2,\\
n  &= &\frac{1}{2}f^2_{y y}=b_{32}x+3b_{33}y+b_{42}x^2+3b_{43}xy+6b_{44}y^2,
\end{array}
\]
so that  
\begin{eqnarray*}
j^2\delta(x,y)=a_{32}x+3a_{33}y+\Big(b_{32}^2-2a_{31}b_{32}+2a_{32}b_{31}+a_{42}+(2b_{31}+3a_{30})a_{32}\Big)x^2\\
+\Big(6b_{33}b_{32}-6a_{31}b_{33}+6a_{33}b_{31}+3a_{43}+(b_{32}+a_{31})a_{32}  +(6b_{31}+9a_{30})a_{33}\Big)xy\\
+3\Big(3b_{33}^2-2a_{32}b_{33}+2a_{33}b_{32}+2a_{44}+  (2b_{32}+a_{31})a_{33}\Big)y^2.
\end{eqnarray*}

The point at surface $M$ when the projection $P_{0}$ have a $P_3(c)$-singularity is called {\it $P_3(c)$-point}. Thus, we have the following result.

\begin{proposition}\label{CurvaParabolica}
 If the origin is a  $P_3(c)$-point, then the parabolic curve can be parametrized by 
 \[
\Big(\frac{6a_{32}b_{33}-9b_{33}^2-6a_{44}}{a_{32}}y^2+o_2(y),y \Big).
\]
\end{proposition}
\begin{proof}  
The $P_3(c)$-point is on the parabolic curve if and only if $a_{33}=0$ and $a_{32}\neq0$. The result follows by applying the implicit function theorem. \qed
\end{proof}

\begin{remark} {\normalfont  The unique asymptotic direction at points on the parabolic curve is transversal to this curve, except at $P_3(c)$-point (\cite{r4surf}). }
\end{remark}

At a $P_3(c)$-point we have  
the following result.

\begin{theorem} \label{P3local}
Assume the origin is a  $P_3(c)$-point. Then:
\begin{enumerate}
\item the set of  $B_2$-singularities  of some orthogonal projection $P_u$ on $M$ is a smooth curve, which we call  $B_2$-curve and can be parametrized by \\
\[
\begin{array}{l}
\displaystyle\left(\frac{2(3a_{32}^3b_{33}-4a_{32}^2b_{33}^2-3a_{44}a_{32}^2-8a_{44}a_{32}b_{33}+12a_{44}b_{33}^2+8a_{44}^2)}{a_{32}(a_{32}-2b_{33})^2}y^2+o_2(y),
y\right).\end{array}
\]
\item   the set of  $S_2$-singularities of some orthogonal projection $P_u$ on $M$ is a smooth curve, which we call  $S_2$-curve  and can be parametrized by\\
\[
{
\begin{array}{l}
\displaystyle\left(\frac{6(a_{32}^3b_{33}+48a_{32}^2b_{33}^2-72a_{32}b_{33}^3-a_{44}a_{32}^2-72a_{44}a_{32}b_{33}+36a_{44}b_{33}^2+24a_{44}^2)}{a_{32}(a_{32}+6b_{33})^2}y^2+o_2(y),
y\right).\end{array}}
\]
\item  the set of  $A_0S_1^{\pm}$-singularities  of some orthogonal projection $P_u$ on $M$ is a smooth curve, which we call  $A_0S_1^{\pm}$-curve and can be parametrized by
\[
\begin{array}{l}
\displaystyle\left(\frac{3a_{32}^2b_{33}^2-4a_{32}a_{44}b_{33}+3a_{44}b_{33}^2+2a_{44}^2}{a_{32}(4a_{32}b_{33}-4b_{33}^2-3a_{44})}y^2+o_2(y),
y\right).\end{array}
\]
\item    the set of  $(A_0S_0)_2$-singularities  of some orthogonal projection $P_u$ on $M$ is a smooth curve, which we call  $(A_0S_0)_2$-curve  and can be parametrized by
\[
\Big(\frac{12a_{32}b_{33}-9b_{33}^2-6a_{44}}{a_{32}}y^2+o_2(y),y \Big).
\]
\item   the set of $A_0S_0|A^{\pm}_1$-singularities   of some orthogonal projection $P_u$ on $M$ is a smooth curve, which we call  $A_0S_0|A^{\pm}_1$-curve and can be parametrized by
\[
\begin{array}{l}
\displaystyle\left(\frac{3a_{32}^2b_{33}^2-16a_{32}a_{44}b_{33}+12a_{44}b_{33}^2+8a_{44}^2}{4(a_{32}b_{33}-b_{33}^2-a_{44})a_{32}}y^2+o_2(y),
y\right).
\end{array}
\]
\end{enumerate}
All the above curves are tangents to parabolic curve at the $P_3(c)$-point, and have generically contact order 2 at the origin.
\end{theorem}
\begin{proof} At a parabolic point we can take $(Q_1,Q_2)=(x^2,xy)$.  As the origin is $P_3(c)$-point we can take the projection $P_0$ with the conditions of the Proposition \ref{CondPar}. 
This singularity is versally unfolded by family of orthogonal projection
\[
P(x,y,\textbf{u})=(x-uy,f^1(x,y)-vy,f^2(x,y)-wy),
\]
if and only if  $5a_{32}b_{33}-6b^{2}_{33}-4a_{44} \neq 0$. To simplify the calculation, we use the following change of coordinates: $(x,y)\mapsto( \bar{x}+u\bar{y},\bar{y})$ and revert to the original notation. Then, we can take the projection in the form 
\[
 P_{\textbf{u}}=(x,f^1(x+uy,y)-vy,f^2(x+uy,y)-wy).
 \] 
We write $(g^1,g^2,g^3)=(x,f^1(x+uy,y)-vy,f^2(x+uy,y)-wy)$. Assume that the point  $(x_1,y_1)$ is a singular point of $P_{\bf u}$. Thus,    
\[
\begin{array}{l}
   w=x_1+b_{31}x_1^2+2b_{32}x_1y_1+3b_{33}y_1^2+2y_1u+o_2(x_1,y_1,u)\\
   v= a_{31}x_1^2+2a_{32}x_1y_1+2ux_1+o_2(x_1,y_1,u). \\
 \end{array}
 \]

1. As $(x_1,y_1)$ is a singularity more degenerate than a cross-cap then, by Theorem \ref{criteria}, 
\[
x_1=\frac{6(a_{32}b_{33}-a_{44})}{a_{32}}y_1^2-\frac{(a_{32}-6b_{33})}{a_{32}}y_1u+\frac{1}{a_{32}}u^2+o_2(y_1,u).
\]
 The $B_2$-singularity  occurs at $(x_1,y_1)$ if $-g^2_{ y y y}g^3_{yy}+g^2_{y y}g^3_{ y y y}=0.$
Using $x_1$ in the last expression we get
\[
12(u+3b_{33}y_1)(2ub_{33}+2b_{33}a_{32}y_1-4a_{44}y_1-a_{32}u)+o_2(y_1,u)=0.
\]
By the implicit function theorem we have $u=\frac{2(-2a_{44}+a_{32}b_{33})}{a_{32}-2b_{33}}y_1+o_2(y_1)$. Hence, we obtain 
\[
x_1=\frac{2(2a_{32}^3b_{33}-2a_{32}^2b_{33}^2-a_{44}a_{32}^2-12a_{44}a_{32}b_{33}+12a_{44}b_{33}^2+8a_{44}^2)}{a_{32}(a_{32}-2b_{33})^2}y_1^2+o_2(y_1)
\]
a smooth curve  and gives the locus of the  $B_2$-singularity at the $P_3(c)$-point on $M$.  
Using change of coordinate    $(\bar{x},\bar{y})\mapsto(x-uy, y)$ 
we have the desired curve on $M$.

2. Consider $x_1$ as in statement  $1.$ The $S_2$-singularity occurs at $(x_1,y_1)$ if
\[
\begin{array}{c}
-g^2_{x x y} (g^3_{yy})^3 + 2g^3_{xy}(g^3_{yy})^2g^2_{x yy}-2g^3_{x y}g^3_{yy}g^2_{y y}g^3_{xy y}+g^2_{y y}g^3_{x x y}(g^3_{y y})^2 -\\ 
- (g^3_{x y})^2g^2_{y y y}g^3_{y y}+(g^3_{x y})^2g^2_{y y}g^3_{y yy}=0.
\end{array}
\]
Thus substituting $x_1$ in the above expression  we obtain
\[
4(u+3b_{33}y_1)(6ub_{33}+18b_{33}a_{32}y_1-12a_{44}y_1+a_{32}u)+o_3(y_1,u)=0.
\]
By the implicit function theorem we have $u=\frac{6y_1(-2a_{44}+3a_{32}b_{33})}{a_{32}+6b_{33}}+o_2(y_1)$. 
Hence, 
{\small\[
x_1=\frac{6(a_{32}^3b_{33}+48a_{32}^2b_{33}^2-72a_{32}b_{33}^3-a_{44}a_{32}^2-72a_{44}a_{32}b_{33}+36a_{44}b_{33}^2+24a_{44}^2)}{a_{32}(a_{32}+6b_{33})^2}y_1^2+o_2(y_1).
\]}
Using the change of coordinate   $(\bar{x},\bar{y})\mapsto(x-uy, y)$ we obtain 
a smooth curve  and gives the locus of the  $S_2$-singularity at the $P_3(c)$-point on $M$.
 
3. By  Proposition \ref{condA0S1} and Theorem \ref{criteria} the $A_0S^{\pm}_1$-singularity occurs if:
\begin{itemize}
  \item[(i)] $P_{\textbf{u}}(x_1,y_1,u)=P_{\textbf{u}}(x_2,y_2,u)$ for $(x_1,y_1)\neq(x_2,y_2)$ and 
  \item[(ii)] $(x_1,y_1)$ is a point of type $S_1$, in Theorem \ref{criteria} we have
  \[
\frac{\partial}{\partial x}det(d_xP_{\textbf{u}},d_yP_{\textbf{u}},d_{yy}P_{\textbf{u}})=0.
\]
 \end{itemize}

Again we can take $x_1$ as in statement $1.$ On the other hand, using the expressions of $w,v$ and $x_1$ in $P_{\textbf{u}}(x_1,y_1,u)=P_{\textbf{u}}(x_2,y_2,u)$  and the implicit function theorem 
\[
{
\begin{array}{l}
   x_2=x_1, \\
   \displaystyle u=-2b_{33}y_1-b_{33}y_2+o_2(x_1,y_1,y_2)\\
    \displaystyle y_2= -\frac{4a_{32}b_{33}-4b_{33}^2-3a_{44}}{a_{32}b_{33}-2b_{33}^2-a_{44}}y_1+o_2(y_1).
 \end{array}}
 \]
Hence, substituting recursively $y_1$ we obtain a smooth curve on $M\times S^3$ given by the expression of $x_2,\;y_2,\;u,\;v,$ and $w$.
Take the initial change of coordinates $(\bar{x},\bar{y})\mapsto (x-uy, y)$ and project this curve by $\pi_1:M\times S^3\rightarrow M$, we have the desired smooth curve that gives the locus of the  $A_0S_1^{\pm}$-singularity at the $P_3(c)$-point on $M$.

4. By Proposition \ref{condA0S1} we have the conditions  of the  bi-germ $(A_0S_0)_2$:
\begin{itemize}
 \item[(i)] $F(x_1,y_1,u)=F(x_2,y_2,u)$ for $(x_1,y_1)\neq(x_2,y_2)$;
  \item[(ii)] The point  $(x_1,y_1)$ is a cross-cap;
  \item[(iii)] The limiting tangent vector to the double points curve of the cross-cap belongs to tangent space to the image of  $F_u$ at point  $(x_2,y_2)$.     
\end{itemize}
The item (ii) determines $v$ and $w$ as already described above. 
Following the similar computations as in item 3 we obtain a smooth curve on  $M\times S^3$.
Take the initial change of coordinates $(\bar{x},\bar{y})\mapsto (x-uy, y)$ and project this curve by $\pi_1:M\times S^3\rightarrow M$, we have a smooth curve that gives the locus of the  $(A_0S_0)_2$-singularity at the $P_3(c)$-point on $M$. 

5. The Proposition \ref{condA0S1} $A_0S_0|A^{\pm}_1$-singularity occurs  if:
\begin{itemize}
  \item[(i)] $F(x_1,y_1,u)=F(x_2,y_2,u)$ for $(x_1,y_1)\neq(y_2,y_2)$;
  \item[(ii)] The point  $(x_1,y_1)$ is a  cross-cap;
  \item[(iii)]  The tangent vector of the cross-cap belongs to the tangent space of image of $F_u$ at point $(x_2,y_2)$. 
  \end{itemize}

 In a similar way as in item 3 we obtain a smooth curve on  $M$.
Finally, by Proposition \ref{CurvaParabolica} and items 1, 2, 3, 4 and 5 we can prove that all curves are tangent to parabolic curve at $P_3(c)$-point 
generically with contact order 2 at the origin. (See Figure \ref{multilocal}). \qed
\end{proof}

\begin{figure}[htp]
\begin{center}
\includegraphics[width=4in, height=4.6cm]{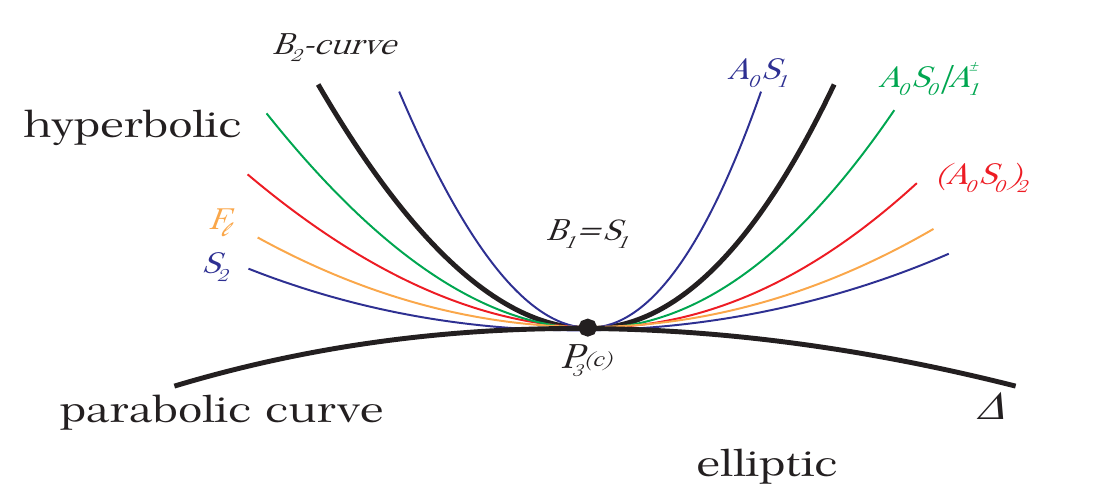}
\caption{Local and multi-local curves tangent to parabolic curve at the $P_3(c)$-point. (The relative position of these curves are given in $\S$ \ref{6}).}
\label{multilocal}
\end{center}
\end{figure}

\section{The BDE on surfaces in $\mathbb R^4$}\label{5}
 
Consider 
a surface $M$ given by Monge form with 2-jet of $(f^1,f^2)$ given by  
$
(Q_1,Q_2)=(ax^2+2bxy+cy^2,lx^2+2mxy+ny^2).
$ 
The asymptotic directions form a field of directions on $M$ and their integral curves are  called {\it asymptotic curves}. The asymptotic curves of the surface are the solutions of the BDEs using the coefficients of $(Q_1,Q_2)$.

Precisely, the asymptotic curves are solutions of the BDE 
\begin{equation}\label{assEDB}
(bn-cm)p^2+(an-cl)p+(am-bl)=0,
\end{equation}
where $p=dy/dx$. 
The {\it discriminant curve} of the BDE  is given by the zeros of the function 
\[
\delta=(an-cl)^2-4(am-bl)(bn-cm)
\]
and coincides with the parabolic set $\Delta$ of $M$.

To study the configurations of asymptotic curves we need some results on BDEs.
Consider the BDE 
\begin{equation}\label{BDE}
\Omega(x, y, p) =A(x, y)dy^2 + 2B(x, y)dxdy + C(x, y)dx^2 = 0.
\end{equation}
The set $\delta=0$, where $\delta = B^2 -AC$, is the discriminant curve of BDE at which the integral curves generically have  cusps. Equation (\ref{BDE}) defines two directions in the plane when $\delta> 0$.  These directions lifts to a single valued field $\xi$ on $\mathcal M=\Omega^{-1}(0)$. A suitable lifted field $\xi$ (see \cite{davbook})  is given by $ \Omega_p\partial y + p\Omega_p\partial x-(\Omega_y + p\Omega x)\partial p$.

One can separate BDE  into two types.
The first case occurs
when the functions $A, B, C$ do not all vanish at the origin.
Then the BDE is just an implicit differential equation (IDE).
The second case is that all the coefficients of BDE vanish at the origin.
Stable topological models of the BDEs belong to the first case;
 it arises when the discriminant is smooth (or empty).
If the unique direction at any point of the discriminant is transverse to it,
then the BDE is smoothly equivalent to
$dx^2+ydy^2=0$, (\cite{cibrario}, \cite{dara}).
If the unique direction is tangent to the discriminant,
then the BDE is stable and smoothly equivalent to
$dx^2+(-x+\lambda y^2)dy^2=0$
with $\lambda\neq 0,\frac{1}{16}$, (\cite{davbook}); 
the corresponding point in the plane is called a \emph{folded singularity} 
-- more precisely, 
a \emph{folded saddle} if $\lambda<0$, 
a \emph{folded node} if $0<\lambda<\frac{1}{16}$ 
and a \emph{folded focus} if $\frac{1}{16}<\lambda$, (see Figure \ref{Folded} and \cite{davbook}). 

A solution curve has an {\it inflection} at the projection of points when $\Omega=\Omega_y+p\Omega_x=0$. There is a smooth curve of such points tangent to the discriminant curve at folded singularities (\cite{duality}).

\begin{figure}[htp]
\begin{center}
\includegraphics[width=5in, height=2cm]{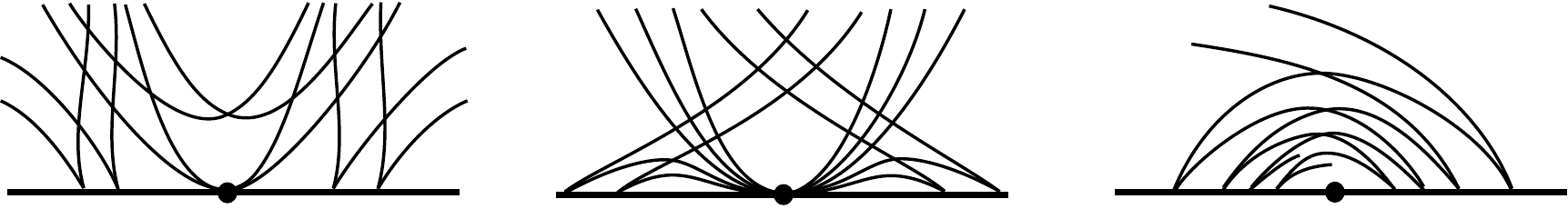}
\caption{A folded saddle (left), node (center) and focus (right).}
\label{Folded}
\end{center}
\end{figure}

For surfaces in $\mathbb R^4$ away from inflection points, the asymptotic curves are generically a family of cusps at ordinary parabolic points and have a folded singularity at a $P_3(c)$-point of projection $P_{\bf u}$ (see \cite{r4surf} and Figure \ref{Folded}).

When the origin is an inflection point the BDE of the asymptotic curves is of type 2, i.e. all its coefficients vanish at the origin (for more detail see \cite{Ronaldoetal} and \cite{r4surf}).

\begin{proposition} \label{FlecCurve}
With the conditions of Theorem $\ref{P3local}$, the 2-jet of the inflections of the asymptotic curves at a $P_3(c)$-point is given by 
\[
\left(\frac{ 6(-a_{44}+a_{32}b_{33})(-36b_{33}^2+24a_{32}b_{33}+a_{32}^2-24a_{44})}{a_{32}(-6b_{33}+a_{32})^2}y^2,y\right).
\]
\end{proposition}

\begin{proof}
The equation  $F(x,y,p)=(am-bl)p^2+(an-cl)p+(bn-cm)=0$ leads to
\[
\begin{array}{c}
-\frac{1}{2}a_{32}x+(-\frac{1}{2}a_{42}+a_{31}b_{32}-a_{32}b_{31})x^2+(3a_{31}b_{33}-\frac{3}{2}a_{43})xy+\\
b_{32}px+(-3a_{44}+3a_{32}b_{33})y^2+3b_{33}yp+\frac{1}{2}p^2+o_2(x,y,p)=0.
\end{array}
\]
Then we can write implicitly  
\[
x= \frac{6(-a_{44}+a_{32}b_{33})}{a_{32}}y^2+6b_{33}{a_{32}}yp+\frac{1}{a_{32}}p^2+o_2(y,p).
\]
Substituting the expression of $x$ in 2-jet of the $F_y+pF_x=0$ we obtain 
\[
\begin{array}{l}
\displaystyle \Big(6a_{32}b_{33}-6a_{44}\Big)y+\Big(3b_{33}-\frac{1}{2}a_{32}\Big)p+\frac{(18a_{31}b_{33}-9a_{43})(a_{32}b_{33}-a_{44})}{a_{32}}y^2+\\
\displaystyle \hspace{0cm}   \Big(\frac{(18a_{31}b_{33}-9a_{43})b_{33}}{a_{32}}+3a_{31}b_{33}-\frac32a_{43}\Big)py+\Big(\frac{6a_{31}b_{33}-3a_{43}}{2a_{32}}+b_{32}\Big)p^2=0\\
\end{array}
\]
Again, solving implicitly we get
\[
p = \frac{12(-a_{44}+a_{32}b_{33})}{(-6b_{33}+a_{32})}y+o_2(y).
\] 
After substituting of $p$ in $x$  we obtain the 2-jet of the curve of inflections of asymptotic curves.  \qed
\end{proof}

The set of inflection points of the asymptotic curves, we called by {\it flecnodal curve}. As a consequence of the Proposition \ref{FlecCurve} the flecnodal curve is tangent to parabolic curve at the $P_3(c)$-point (see the yellow curve on Figure \ref{multilocal}).

\section{Cross-ratio invariants at a $P_3(c)$-point}\label{6}

\indent \indent The $P_3(c)$-points have similar behavior   to cusps of Gauss of surfaces in  $\mathbb R^3$. In fact, the asymptotic curves of a surface in $\mathbb R^3$ (resp. $\mathbb R^4$) has fold singularity at a cusp of Gauss (resp. at $P_3(c)$-point).

By projective classification in \cite{DK}, the $4$-jet of a parametrization of a surface $M$ in Monge form $(z,w)=(f_1(x,y),f_2(x,y))$  at a $P_3(c)$-point is equivalent, by projective transformations, to  
\begin{equation}\label{P3forma}
(z,w)=(x^2 +xy^2 + {\alpha} y^4,xy + {\beta} y^3 +\phi),
\end{equation}
where $6{\beta}^2 + 4{\alpha}-15{\beta} + 5\neq0,$ ${\alpha}\neq 0, 1/2, 1, 3/2$
and $\phi$ is a polinomial in $x,y$ of degree $4$.

This allows to recalculate the expressions of the curves local and multi-local singularities  at the $P_3(c)$-point.
 
\begin{proposition}\label{curvas} With notation as above and for the surface with $4$-jet as in $(\ref{P3forma})$, we have the following curves passing through the  $P_3(c)$-point:

\begin{itemize}
\item[(a)] The parabolic curve, denoted by $\Delta$-curve, is a smooth curve given by 
\[
x=c_Py^2+ o_2(y),
\]
with $c_P=3(-3\beta^2  - 2 \alpha + 2 \beta)$.
\item[(b)] The $B_2$-curve is a smooth curve given by
\[
x=c_By^2+ o_2(y),
\]
with $c_B=\displaystyle\frac{2(12 \alpha \beta^2  + 8 \alpha^2  - 8\alpha \beta - 4 \beta^2  -  3\alpha + 3 \beta)}{(2\beta-1)^2}$.
\item[(c)] The $S_2$-curve is a smooth curve given by
\[
x=c_Sy^2+ o_2(y),
\]
with $c_S=\displaystyle\frac{6(36 \alpha \beta^2  -72\beta^3 + 24 \alpha^2  - 72 \alpha \beta + 66 \beta^2  -  \alpha + \beta)}{(6\beta+1)^2}$.
\item[(d)] The flecnodal curve, denoted by $F_l$-curve, is a smooth curve given by
\[
x=c_Fy^2+ o_2(y),
\]
with $c_F=\displaystyle\frac{6(-36 \beta^2  +24\beta+1- 24 \alpha)(-\alpha + \beta)}{(-6\beta+1)^2}$.
\item[(e)] The $(A_0S_0)_2$-curve is a smooth curve given by
\[
x=c_{s_{02}}y^2+ o_2(y),
\]
with $c_{s_{02}}=3(-3\beta^2  - 2 \alpha + 4 \beta)$.
\item[(f)] The $A_0S_1$-curve, is a smooth curve given by
\[
x=c_{s1}y^2+ o_2(y),
\]
with $c_{s_1}=\displaystyle-\frac{3 \alpha \beta^2  + 2 \alpha^2  - 4 \alpha \beta + 3\beta^2}{-4 \beta^2  -3\alpha +4\beta}$.
\item[(g)]  The $A_0S_0|A_1$-curve, is a smooth curve given by
\[
x=c_{s_{01}}y^2+ o_2(y),
\]
with $c_{s_{01}}=\displaystyle\frac{1}{4}\frac{(3 \beta^2  -16\alpha\beta + 8 \alpha^2  + 12\alpha\beta^2)}{ - \beta^2  - \alpha + \beta}$.
\end{itemize}
\end{proposition}
\begin{proof}All the singularities are projective invariants, then we can consider representing $M$ locally as a surface  $\bar{M}$ in $\mathbb P^4$ given in affine chart  $\{[x:y:z:w:1]\}$ in Monge form $[x:y:f^1(x,y):f^2(x,y):1]$. We can take  $(f^1,f^2)$ in normal form (\ref{P3forma}) and use the equations of the curves in  Theorem \ref{P3local} and Proposition \ref{FlecCurve} with   $a_{32}=1$, $a_{44}=\alpha$ e  $b_{33}=\beta$. \qed  \end{proof}

\vspace{0.5cm}

Let $PT^*M$ be the projective cotangent bundle with the canonical contact structure. Take an affine chart and  identify $PT^*M$ with $\mathbb R^3$ in coordinates $(x,y,p)$, where the contact structure   is given by   $dy-pdx$. Consider a curve   $\gamma$ in $M$. Its the Legendrian lift in $PT^*M$ is the curve $(\gamma(t),\left[\frac{dy}{dx}\right])$.

We denote the tangent lines  to the Legendrian lifts of the curves $\Delta$, $B_2$, $S_2$, $F_l$,  $(A_0S_0)_2$, $A_0S_1$ and $A_0S_0|A^{\pm}_1$ in $PT^*M$ at  $P_3(c)$-point  by $l_P, \;l_B, \;l_S,\;l_{F},\; l_{s_{02}},\; l_{s_1},$ and   $l_{s_{01}}$, respectively. Also by $l_g$ we mean the element of contact of the $P_3(c)$-point (the vertical line in the contact plane at the $P_3(c)$-point). By Proposition \ref{curvas}, all theses lines belong to the contact plane at  the $P_3(c)$-point and we write the 2-jets of the curves above as $y = c_P x^2$, $y =c_B x^2$, $y=c_Sx^2$, $y=c_{F}x^2$, $y=c_{s_{02}}x^2$, $y= c_{s_1}x^2$,   $y=c_{s_{01}}x^2$ and $y=c_gx^2$, respectively.

Recall that the cross-ratio of four coplanes concorrent lines $l_i$, $i=1,\ldots,4$ with respective gradients  $c_i$, $i =1,\ldots,4$ is given by  
\[
(l_1,l_2:l_3,l_4)=\frac{c_3-c_1}{c_3-c_2} \cdot \frac{c_4-c_2}{c_4-c_1}.
\]
Observe that if $l_2$ is vertical, then 
\[
(l_1,l_2:l_3,l_4)=\frac{c_3-c_1}{ c_4-c_1}.
\]

\begin{remark} {\normalfont For a surface in $\mathbb R^3$, Uribe-Vargas introduce in \cite{uribeInv} the cr-invariant of the surface at a cusp of Gauss. The main idea is to use the parabolic curve, the conodal curve (curve of the multi-local  singularity  $A_1A_1$ of the height function) and the flecnodal curve. Such curves are tangent at cusps of Gauss. Uribe-Vargas lifted theses curves to the projective cotangent bundle of $M$. The tangents lines and one more vertical line (the fiber of $p$) 
belongs to the same plane and their  cross-ratio is the cr-invariant.
In our case, we have a  wide choice of curves at a $P_3(c)$-point that can be used to obtain cross-ratios. 
}
\end{remark}

\begin{definition} The  {\it cr-invariants} $\rho$ at a $P_3(c)$-point of a surface in $\mathbb R^4$ are defined as the cross-ratios of any four of the tangent lines of the Legendrian lifts of the curves in  Proposition \ref{curvas}.\end{definition}


\begin{theorem}\label{principal} Let $M$ be a surface in $\mathbb P^4$ given locally by  Monge form  $(f^1,f^2)$ with the $P_3(c)$-point being the origin. Then there are three $cr$-invariants 
$\rho_1$,  $\rho_2$ and $\rho_3$ that allow can be used to recover the projective invariants $\alpha$ and $\beta$ of the normal form  $(\ref{P3forma})$. We then have 
\[ 
{\beta}=\frac{\rho_1-1}{3(2\rho_1-1)},
\] 
\[ 
{\alpha}=\frac{(80\rho_2\rho_1-32\rho_2+20\rho_3\rho_1-8\rho_3+42\rho_1-21)(\rho_1-1)}{9(3\rho_3+1+12\rho_2)(2\rho_1-1)^2},
\]
where $\rho_1=(l_P, l_B: l_S,l_{F})$, $\rho_2=(l_P,l_g:l_{s_{01}},l_{s_{02}})$ and $\rho_3=(l_P,l_g:l_{s_{1}},l_{s_{02}})$.
\end{theorem}
\begin{proof}
 Consider the  cross-ratio $\rho_1=(l_P, l_B: l_S,l_{F})$ which is given by 
\[ 
\beta=\frac{\rho_1-1}{3(2\rho_1-1)}.
\] 
Therefore with the local singularities we can get the projective invariant $\beta$.
Consider the cr-invariants $\rho_2=(l_P,l_g:l_{s_{01}},l_{s_{02}})=\frac{1}{24}\frac{21\beta^2-32\beta\alpha+48\alpha\beta^2+16\alpha^2-60\beta^3+36\beta^4}{\beta(-\beta+\beta^2+\alpha)}$ and $\rho_3=(l_P,l_g:l_{s_{1}},l_{s_{02}})=\frac{1}{6}\frac{16\alpha^2-38\beta\alpha+48\alpha\beta^2+21\beta^2-60\beta^3+36\beta^4}{\beta(3\alpha-4\beta+4\beta^2)}$. We obtain the following  equations 
\[
\left\{\begin{array}{l}
-16\alpha^2+(18\rho_2\beta-48\beta^2+38\beta)\alpha-36\beta^4-24\rho_2\beta^2+24\rho_2\beta^3+60\beta^3-21\beta^2=0 \\
16\alpha^2+(72\rho_3\beta+48\beta^2-32\beta)\alpha+36\beta^4-96\rho_3\beta^2+96\rho_3\beta^3-21\beta^2-60\beta^3=0.
\end{array}\right.
\]
Adding the two equations gives
\[
(18\rho_2\beta+6\beta+72\rho_3\beta)\alpha-24\rho_2\beta^2+24\rho_2\beta^3-42\beta^2-96\rho_3\beta^2+96\rho_3\beta^3=0.
\] 
Substituting $\beta$ by its expression gives 
\[ 
\alpha=\frac{(80\rho_2\rho_1-32\rho_2+20\rho_3\rho_1-8\rho_3+42\rho_1-21)(\rho_1-1)}{9(3\rho_3+1+12\rho_2)(2\rho_1-1)^2}.
\] 
 \qed \end{proof}

We consider now the folded singularity of asymptotic curves at a $P_3(c)$-point and determine their types in term of the parameters $\alpha$ and $\beta$ in (\ref{P3forma}). 

\begin{theorem}\label{asym} Suppose that $M$ is given as in $(\ref{P3forma})$ at a $P_3(c)$-point. The asymptotic curves have a folded singularity if $-6\beta^2-4\alpha+5\beta\neq0,\frac{1}{24}$. The singularity is a folded saddle if  $-6\beta^2-4\alpha+5\beta<0$, a folded node if  $0<-6\beta^2-4\alpha+5\beta<\frac{1}{24}$ and a folded focus if  $-6\beta^2-4\alpha+5\beta>\frac{1}{24}$.
\end{theorem}
\begin{proof}
Denote by $\Omega(x,y,p)=0$ the BDE of asymptotic curves given by (\ref{assEDB}) then 
\[
j^2\Omega(x,y,p)=\frac12p^2+3\beta yp+(-\frac12x+(-3\alpha+3\beta)y^2).
\]
The linear part of the projection to plane $(y, p)$ of the lifted field $\xi$ associated to $\Omega$ is 
\[
(3\beta y+p) \frac{\partial}{\partial y}+\Big(6(\alpha-\beta)+(\frac12-3\beta)p\Big)\frac{\partial}{\partial p}.
\]
The eigenvalues of the matrix associated to the above linear vector are 
\[
\frac12\pm\sqrt{\frac{1}{16}-(-9\beta^2-6\alpha+\frac{15}{2}\beta)},
\]
and the result follow. Figure \ref{Folded} shows the configuration of asymptotic curves at a folded singularity. \qed
\end{proof}

\begin{remark}{\normalfont  If $-6\beta^2-4\alpha+5\beta=0$ the family of the orthogonal projection is not a versal unfolding of the $P_3(c)$-point, i.e., this case is non-generic. Also, if $-6\beta^2-4\alpha+5\beta=0,\frac{1}{24}$, in a generic one-parameter family of surfaces, the asymptotic curves undergo  some bifurcations (see, for example, \cite{r4surf}).
}\end{remark}

The Proposition \ref{curvas} shows that the  2-jets of the curves of local and multi-local singularity of $P_u$ at a $P_3(c)$-point are projectively invariants and they depend only on  $\alpha$ and $\beta$. As the generic relative position of these curves at a $P_3(c)$-point is determined by  2-jets of these curves, we can determine these position in terms of $\alpha$ and $\beta$. The relative positions of the curves  $\Delta$, $B_2$, $S_2$ and $F_l$ in terms of  $\alpha$ and $\beta$ are as follows.

\begin{theorem}
With above notation, there are $4$ possible relative positions of the curves  $\Delta$, $B_2$, $S_2$, $F_l$:
\begin{center}
\begin{itemize}
\item[(i)] If $\beta<0$, then $c_P<c_B<c_F<c_S$
\item[(ii)] If $0<\beta<1/6$, then $c_P<c_B<c_S<c_F$
\item[(iii)] If $1/6<\beta<1/3$, then $c_P<c_S<c_B<c_F$
\item[(iv)] If $\beta>1/3$,  then $c_P<c_S<c_F<c_B.$
 \end{itemize}
 \end{center}
\end{theorem}
\begin{proof} Consider the coefficients  $c_P$, $c_B$, $c_S$ and $c_F$ as in Proposition  \ref{curvas}. We have   
\[
c_P-c_B=-\frac{(4\alpha+6\beta^2-5\beta)^2}{(2\beta-1)^2}<0, \mbox{ for any value of $\alpha$ and $\beta$};
\]
\[
c_P-c_S=-\frac{9(4\alpha+6\beta^2-5\beta)^2}{(1+6\beta)^2}<0, \mbox{ for any value of $\alpha$ and $\beta$};
\]
\[
c_P-c_F	=-\frac{9(4\alpha+6\beta^2-5\beta)^2}{(6\beta-1)^2}<0, \mbox{ for any value of $\alpha$ and $\beta$},
\]
where $4\alpha+6\beta^2-5\beta\neq0$ is the condition for the $\mathcal A_e$-versality at the $P_3(c)$-point (Proposition \ref{CondPar}). Furthermore, 
\[
c_B-c_S=\frac{8(6\beta-1)(4\alpha+6\beta^2-5\beta)^2}{(2\beta-1)^2(1+6\beta)^2}.
\] 
So, $c_B>c_S$ if and only if  $\beta>1/6$. We have
\[
 c_B-c_F=\frac{8(3\beta-1)(4\alpha+6\beta^2-5\beta)^2}{(2\beta-1)^2(1+6\beta)^2},
\] 
thus, $c_B>c_F$ if and only if  $\beta>1/3$. Also one can prove that
\[
c_S-c_F=-\frac{216\beta(4\alpha+6\beta^2-5\beta)^2}{(6\beta-1)^2(1+6\beta)^2}.
\] 
Therefore, $c_S>c_F$ if and only if $\beta<0$. \qed
\end{proof}

\vspace{0.5cm}

The relative positions of the multi-local curves and parabolic curve at $P_3(c)$-point is mention in next Theorem.

\begin{theorem} With notation as above,  there are $22$ possibilities for the relative positions of the curves  $\Delta$, $(A_0S_0)_2$, $A_0S_1$, $A_0S_0|A_1^\pm$. Each region in the $(\alpha,\beta)$-plane in Figure $\ref{posiMulti}$ is determined by the following conditions:  
\[
\begin{array}{rllrl}
1:& c_P<c_{s_{02}}< c_{s_{01}}<c_{s_1}&\;\;\;\;\;\;\;\;&12:& c_P<c_{s_{02}}<c_{s_1}< c_{s_{01}}\\
2:& c_{s_{02}}<c_P<c_{s_{01}}<c_{s_1}&&13:& c_P<c_{s_1}<c_{s_{01}}<c_{s_{02}}\\
3:& c_{s_{02}}<c_P<c_{s_1}<c_{s_{01}}&&14:& c_P<c_{s_{01}}<c_{s_{02}}<c_{s_1}\\
4:& c_{s_{02}}<c_{s_{01}}<c_{s_1}<c_P&&15:& c_{s_{01}}<c_P<c_{s_{02}}<c_{s_1}\\
5:& c_{s_1}<c_{s_{01}}<c_{s_{02}}<c_P&&16:& c_P<c_{s_1}<c_{s_{02}}<c_{s_{01}}\\
6:& c_{s_1}<c_{s_{01}}<c_P<c_{s_{02}}&&17:& c_{s_{01}}<c_{s_1}<c_P<c_{s_{02}}\\
7:& c_{s_1}<c_P<c_{s_{01}}<c_{s_{02}}&&18:& c_P<c_{s_{01}}<c_{s_1}<c_{s_{02}}\\
8:& c_P<c_{s_1}<c_{s_{01}}<c_{s_{02}}&&19:& c_P<c_{s_{01}}<c_{s_1}<c_{s_{02}}\\
9:&c_P<c_{s_{02}}<c_{s_{01}}<c_{s_1}&&20:&c_{s_{01}}<c_P<c_{s_1}<c_{s_{02}}\\
10:& c_P<c_{s_1}<c_{s_{01}}<c_{s_{02}}&&21:& c_{s_{01}}<c_{s_{02}}<c_{s_1}<c_P\\
11:& c_P<c_{s_{01}}<c_{s_{02}}<c_{s_1}&&22:& c_{s_{02}}<c_{s_{01}}<c_P<c_{s_1}.\\
 \end{array}
 \] 
The curves $\gamma_i$,  $i=1,\ldots,6$  in Figure $\ref{posiMulti}$ are given by   
\[
\left\{\begin{array}{l}
\gamma_1:\beta=0,\\ 
\gamma_2:48\beta^2\alpha-60\beta^3+36\beta^4-38\beta\alpha+21\beta^2+16\alpha^2=0,\\
\gamma_3:-60\beta^3+36\beta^4+48\beta^2\alpha+21\beta^2-32\beta\alpha+16\alpha^2=0,\\
\gamma_4: 4\alpha+6\beta^2-5\beta=0,\\ 
\gamma_5: 4\alpha+6\beta^2-9\beta=0, \\
\gamma_6: \alpha=0.\\
 \end{array}
 \right.
 \]
\begin{figure}
  \centering
  \includegraphics[width=10cm]{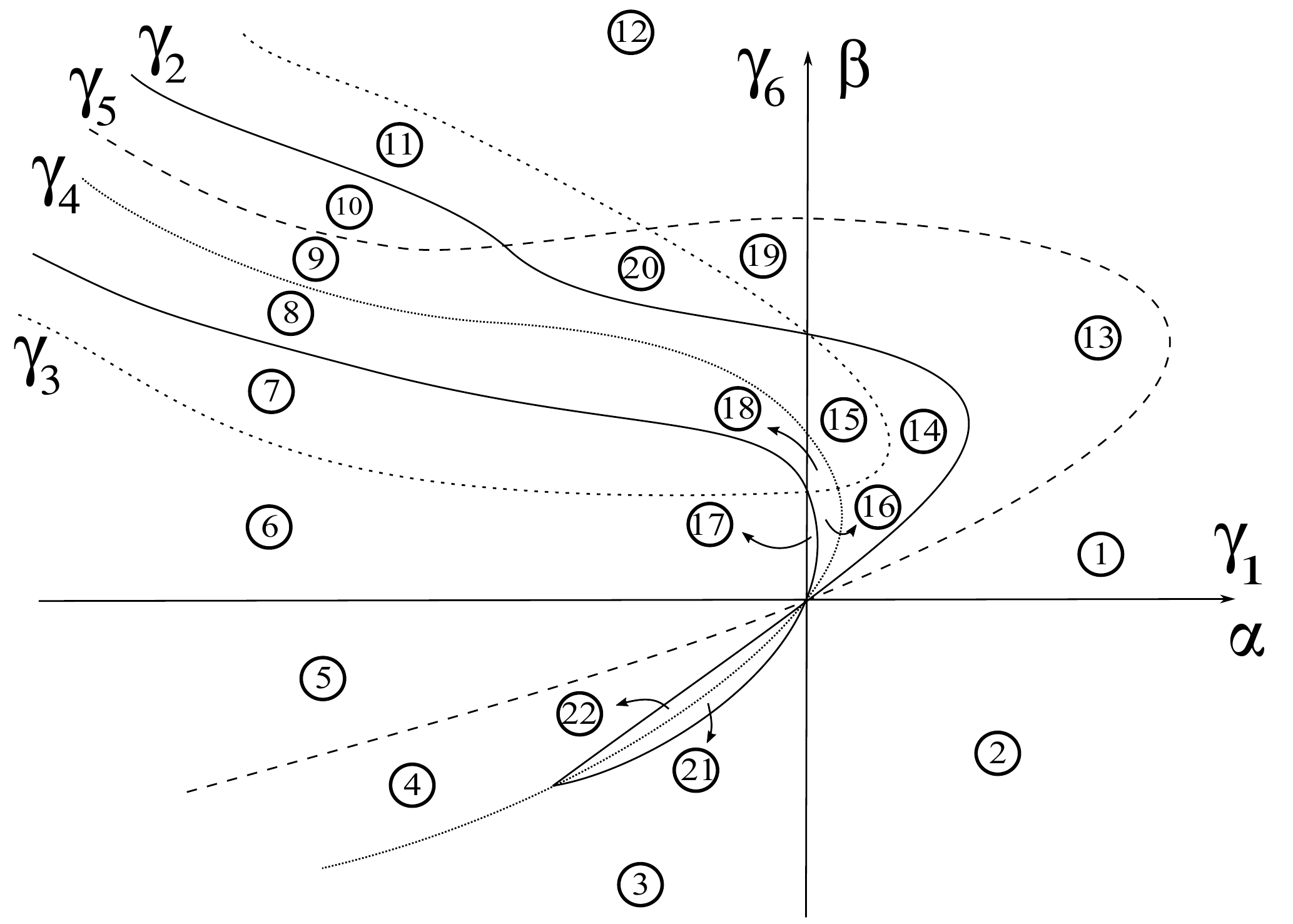}
  \caption{Partition of the $(\alpha,\beta)$-plane into region with a given relative position of the curves of multi-local singularities together with $\Delta$ at a $P_3(c)$-point.}\label{posiMulti}
 \end{figure}
\end{theorem}
\begin{proof}
It follows from Proposition \ref{curvas} that
\[
\begin{array}{l}
\displaystyle c_P-c_{s_{02}}=-6\beta,\\
\displaystyle c_P-c_{s_1}=-\frac{48\beta^2\alpha-60\beta^3+36\beta^4-38\beta\alpha+21\beta^2+16\alpha^2}{3\alpha-4\beta+4\beta^2},\\
\displaystyle c_P-c_{s_{01}}=-\frac{-60\beta^3+36\beta^4+48\beta^2\alpha+21\beta^2-32\beta\alpha+16\alpha^2}{4(-\beta+\beta^2+\alpha)},\\
\displaystyle c_{s_{02}}-c_{s_1}=-\frac{(4\alpha+6\beta^2-9\beta)(4\alpha+6\beta^2-5\beta)}{3\alpha-4\beta+4\beta^2},\\
\displaystyle c_{s_{02}}-c_{s_{01}}=-\frac{(4\alpha+6\beta^2-9\beta)(4\alpha+6\beta^2-5\beta)}{4(-\beta+\beta^2+\alpha)},\\
\displaystyle c_{s_1}-c_{s_{01}}=\frac{\alpha(4\alpha+6\beta^2-9\beta)(4\alpha+6\beta^2-5\beta)}{4(3\alpha-4\beta+4\beta^2)(-\beta+\beta^2+\alpha)}.\\
\end{array}
\]
The curves described in Figure \ref{posiMulti} are given by the vanishing of the above differences (they give the locii of point in the $(\alpha,\beta)$-plane where the relevant curves have contact greater than $2$).  \qed
\end{proof}

\vspace{0.5cm}

Now we present the configurations of the curves of the local singularities that pass through the $P_3(c)$-point in relation to the parabolic curve.

\begin{theorem} The configurations of curves  $\Delta$ and $B_2$ (resp. $\Delta$ and $S_2$, and $\Delta$ and the flecnodal curve) at the $P_3(c)$-point are as shows in Figure $\ref{Conf B}$, (resp $\ref{Conf S}$ and $\ref{Conf F}$).
\end{theorem}
\begin{proof} Consider the 2-jets of the parametrizations of the curves $\Delta$, $B_2$, $S_2$ and of the flecnodal curve and coefficients  $c_P$, $c_B$, $c_S$ and $c_F$ as Proposition \ref{curvas}.
The generic configurations of the curves occur when we avoid that  $c_P=c_B=c_S=c_F=0$. 
We have 3 cases:
\begin{itemize}
\item[(a)] Configurations of the curves $\Delta$ and $B_2$:

The vanishing of  $c_P$ and $c_B$ give algebraic curves (see Figure \ref{Conf B}).  
The curve $c_B=0$  have a singularity when  $\alpha=1/4$ and $\beta=1/2$. Furthermore, the intersection of the curves occurs in a neighborhood of  the origin. The generic configurations of the curves occur when $c_P$ and $c_B$ are different  from zero, i.e., in regions represented by numbers 1, 2, 3, 4, 5 and 6 in Figure \ref{Conf B}.
\item[(b)] Configurations of the curves $\Delta$ and $S_2$:

The vanishing of  $c_P$ and $c_S$ give algebraic curves (see Figure \ref{Conf S}). The curve $c_S=0$ have 3 singularities when 
$(\alpha,\beta)=(1/4,-1/6),(\frac{-79-17\sqrt{17}}{96}, \frac{-5-\sqrt{17}}{12}), \\ (\frac{-79+17\sqrt{17}}{96},\frac{-5+\sqrt{17}}{12})$.
Furthermore,  the intersection of the curves occur in a  neighborhood of the origin. The generic configurations of the curves occur when $c_P$ e $c_S$ are different  from zero, i.e., in regions represented by numbers 1, 2, 3, 4, 5 and 6 in Figure \ref{Conf S}.
\item[(c)] Configurations of the curves $\Delta$ and flecnodal:

The vanishing of $c_P$ and $c_F$ give algebraic curves (see Figure \ref{Conf F}). The curve $c_F=0$ have 3 singularities when 
$(\alpha,\beta)=(1/48,0),(\frac{1}{6},\frac{1}{6}),(-\frac{1}{6},-\frac{1}{6})$.
Furthermore,  the intersection of the curves occurs only at neighborhood of  the origin. The generic configurations of the curves occur when $c_P$ and $c_F$ are different  from zero, i.e., in regions represented by numbers 1, 2, 3, 4, 5, 6  and 7 in Figure \ref{Conf F}.
\end{itemize} 
\end{proof}

\begin{figure}[b]
\centering
\begin{minipage}{65mm}
\includegraphics[width=5.0cm,height=8.2cm]{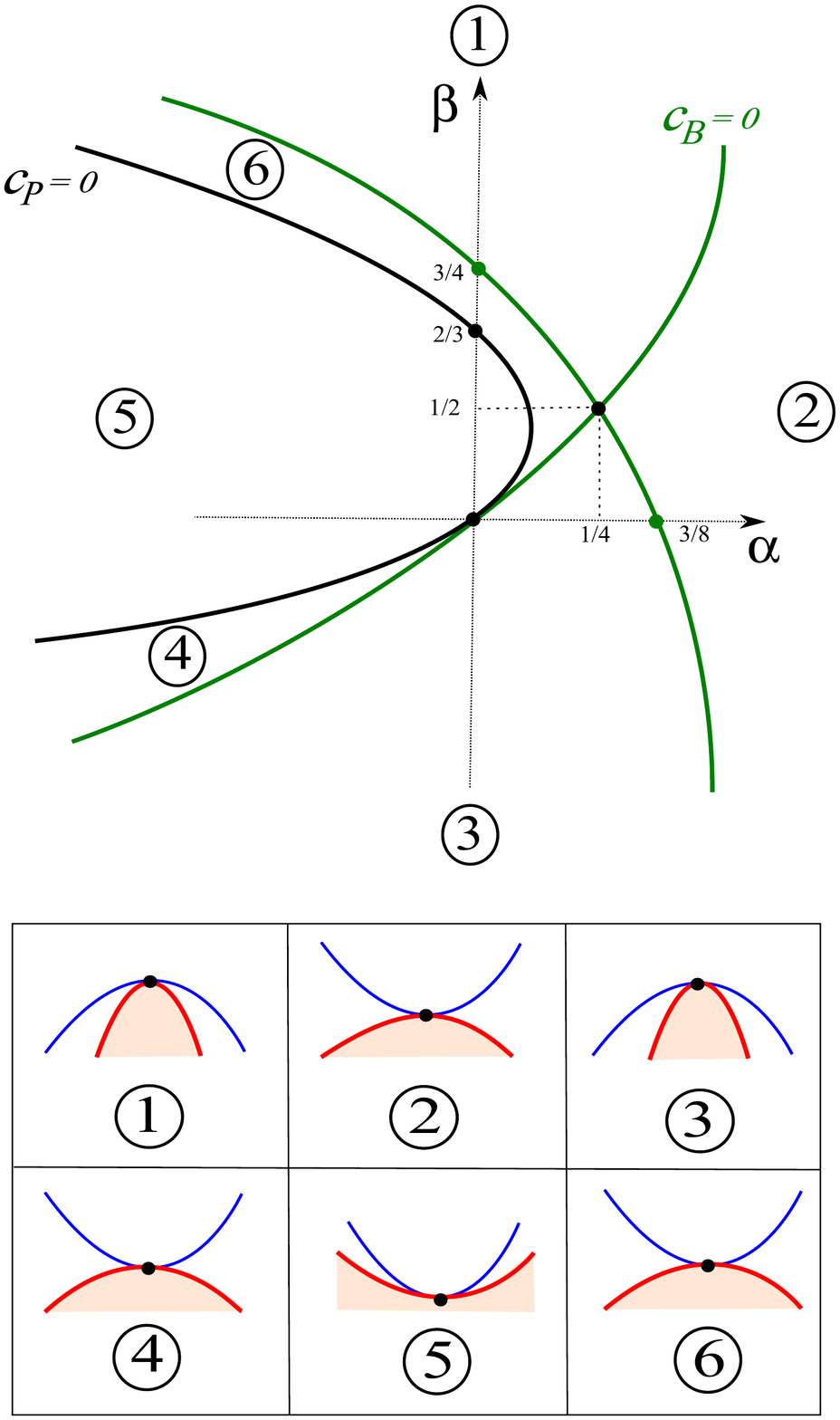}
\caption{Configurations of the curves $\Delta$ and $B_2$ at a $P_3(c)$-point.}
\label{Conf B}
\end{minipage}
\hfil
\begin{minipage}{65mm}
\includegraphics[width=5.0cm,height=8.2cm]{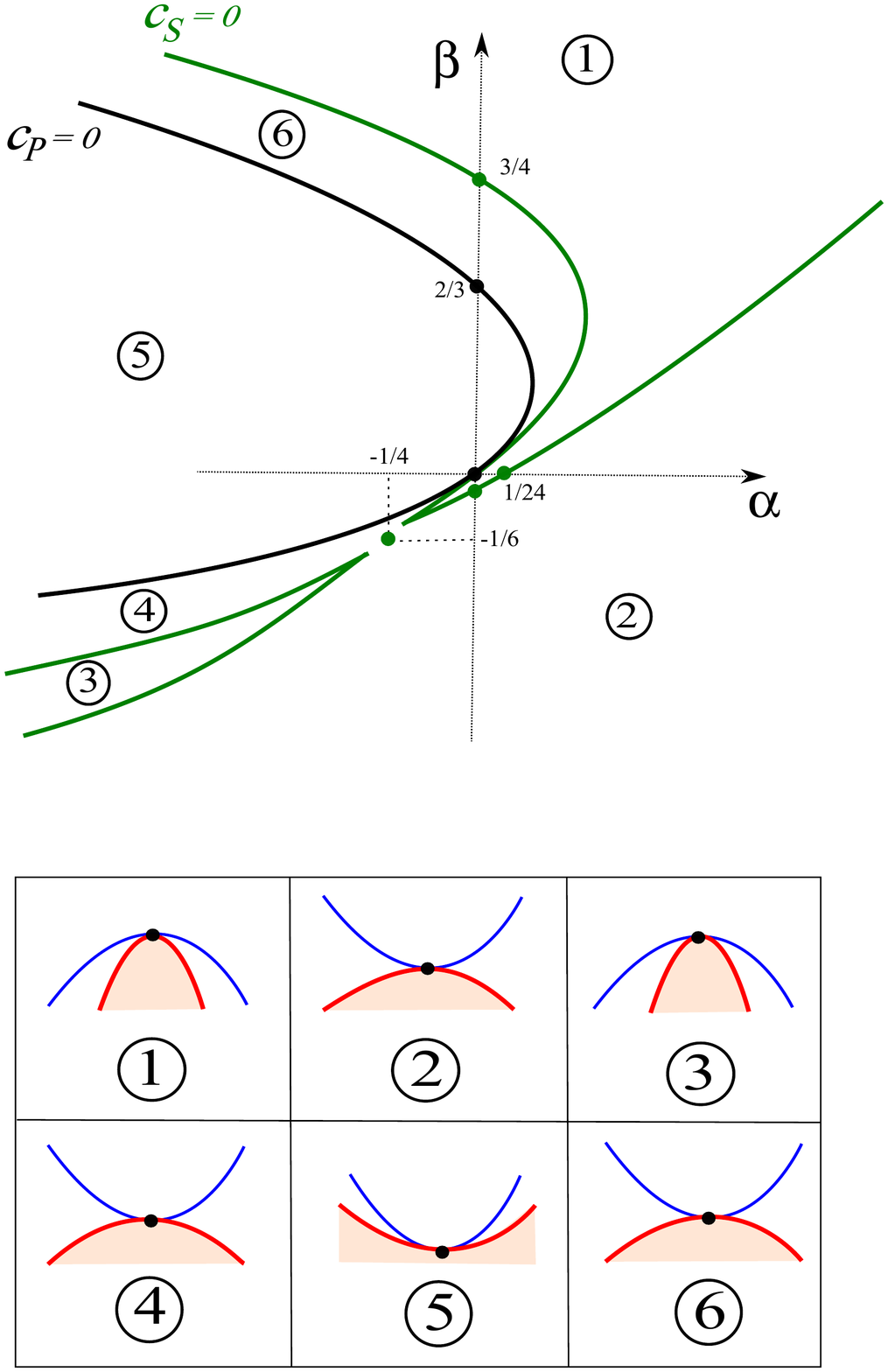}
\caption{Configurations of the curves  $\Delta$ and $S_2$ at a $P_3(c)$-point.}
\label{Conf S}
\end{minipage}
\end{figure}

\begin{figure}[b]
\centering
\begin{minipage}{65mm}
\includegraphics[width=5.6cm,height=7cm]{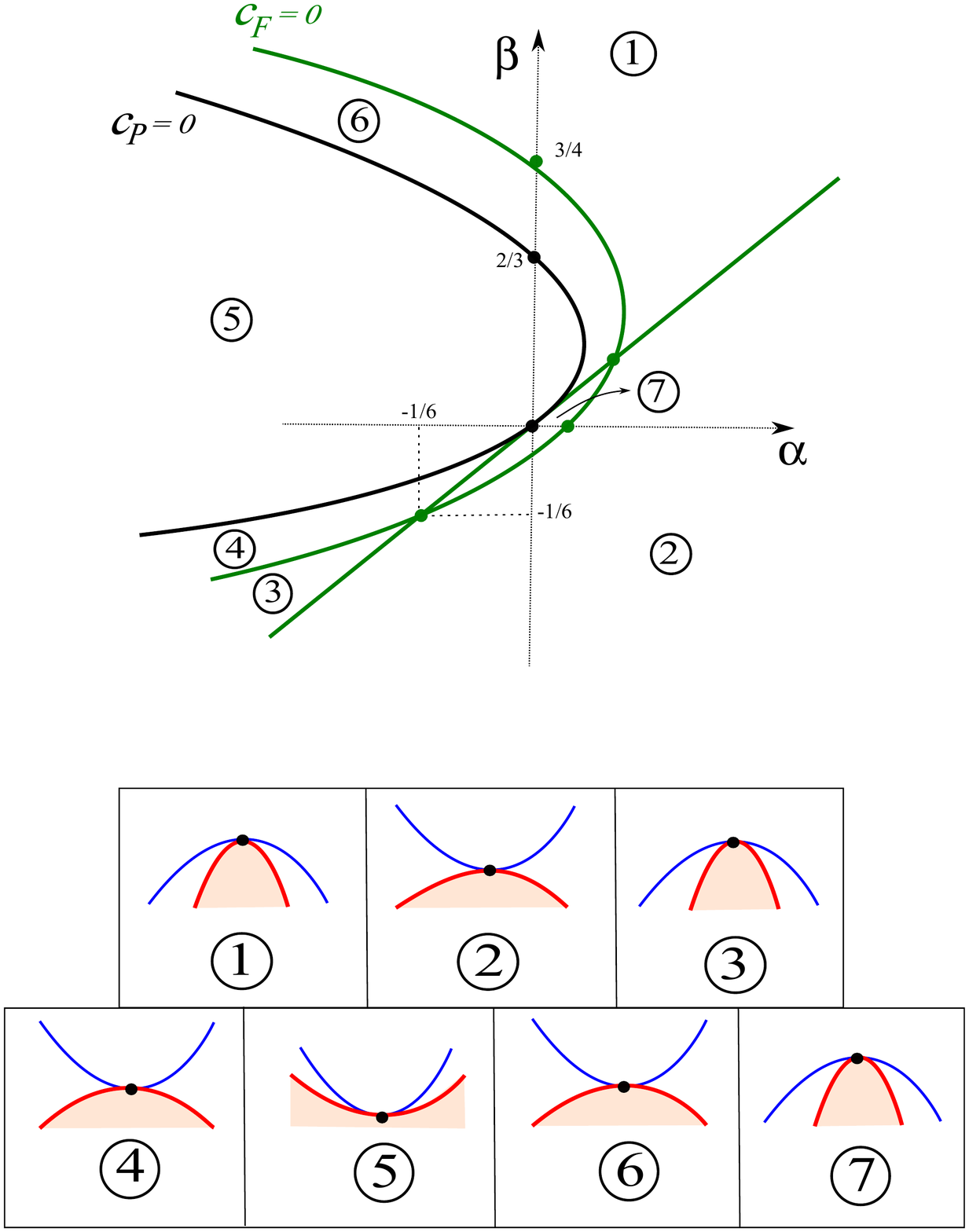}
\caption{Configurations of the curve $\Delta$ and of the flecnodal curve at a $P_3(c)$-point.}
\label{Conf F}
\end{minipage}
\end{figure}


\hspace{-0,5cm}{\bf Acknowledgement}
The author would like to thank Farid Tari and Toru Ohmoto for supervision and comments, respectively. 
The author was totally supported by FAPESP grants no.2012/00066-9.


\end{document}